\documentclass{article}	
\usepackage{graphicx}
\usepackage[utf8]{inputenc}
\usepackage{indentfirst,csquotes}

\usepackage[bottom=1.4in]{geometry}
\usepackage[english]{babel}
\usepackage{tikz}
\usepackage{fancyhdr} 
\usepackage{amsthm}
\usepackage{nicematrix}
\usepackage{dirtytalk}
\usepackage{hyperref}
\usepackage{amssymb}
\usepackage{mathtools}
\usepackage{amsmath}
\usepackage{latexsym}
\usepackage{esint}
\usepackage{wasysym}
\usepackage{pst-eucl}
\usepackage{comment}
\usepackage{hyperref}
\usepackage{blkarray}
\hypersetup{
    colorlinks=true,
    linkcolor=blue,
    citecolor=teal
    }

\newcommand{\ps}[2]{\left<#1,#2\right>}

\newcommand{\norm}[1]{\left\lVert#1\right\rVert}
\newcommand{\nor}[1]{\left\lvert#1\right\rvert}

\newtheorem{Theorem}{Theorem}
\newtheorem*{theoremA}{Theorem A}
\numberwithin{Theorem}{section}
\newtheorem{Proposition}[Theorem]{Proposition}
\newtheorem{Lemma}[Theorem]{Lemma}
\newtheorem{Corollary}[Theorem]{Corollary}
\theoremstyle{definition}
\newtheorem{Definition}[Theorem]{Definition}
\theoremstyle{remark}
\newtheorem{Remark}[Theorem]{Remark}
\numberwithin{equation}{section}

 \DeclareMathOperator*{\dist}{dist}
 
 \DeclareMathOperator{\R}{\mathbb{R}}

\DeclareMathOperator*{\tr}{tr}

\newcommand{\Mm}{\mathcal{M}_{K}^{-}}
\newcommand{\p}{\partial}

\newcommand{\osc}{\operatorname{osc}}

\begin{document}

\title{On the Geometry of Solutions of the Fully Nonlinear Inhomogeneous One-Phase Stefan Problem}

\author{Fausto Ferrari
  \and
  Davide Giovagnoli
  \and 
  David Jesus
}

\newcommand{\Addresses}{{
  \bigskip
  \footnotesize
Fausto Ferrari, Davide Giovagnoli, David Jesus \\
  \textit{E-mail addresses:} {\tt fausto.ferrari@unibo.it, d.giovagnoli@unibo.it, david.jesus2@unibo.it}\\
  \textsc{Dipartimento di Matematica,
Universit\`a di Bologna\\ 
Piazza di Porta San Donato 5, 40126 Bologna, Italy}
 }
}

\date{\today}

\maketitle
{\bf Abstract.} In this paper, we characterize the geometry of solutions to  one-phase inhomogeneous fully nonlinear Stefan problem  with flat free boundaries under a new nondegeneracy assumption.
This continues the study of regularity of flat free boundaries for the linear inhomogeneous Stefan problem started in \cite{FFGJ}, as well as justifies the definition of flatness assumed in \cite{wang2024free}.

\thispagestyle{fancy}
\fancyhead{} 
\fancyfoot{}
\fancyfoot[L]{\footnoterule {\small  
}}

\section{Introduction}
In this paper we obtain geometric results concerning the free boundary of solutions of a fully nonlinear one-phase Stefan problem governed by an inhomogeneous equation. That is, for $\Omega \subset \mathbb{R}^{n}$, $T > 0$, we deal with nonnegative viscosity solutions $u: \Omega \times (0,T] \mapsto \mathbb{R}$ of the following  free boundary problem
    \begin{equation}\label{eq:StefanNH}
        \begin{cases}
          \partial_{t} u - F(D^2 u)  =  f& \text{in} \,\, \Omega^+(u):=\left(\Omega \times (0,T] \right) \cap \{ u > 0  \},\\
       \partial_{t} u = \nor{\nabla u}^2& \text{on} \,\, \mathcal{F}(u):=\left(\Omega \times (0,T] \right) \cap \partial\{ u > 0  \},
    \end{cases}
    \end{equation}
   where $F$ is a uniformly elliptic operator, $ f \in C(\Omega \times [0,T]) \cap L^{\infty }(\Omega \times [0,T])$ and \(\mathcal{F}(u)\) is the \textit{free boundary} of the solution. More precisely we relate the geometric condition of free boundary flatness in the evolutive framework to a suitable notion of solution flatness in the inhomogeneous case. This result, together with the methods introduced in \cite{DFS23} and further developed in \cite{FFGJ,wang2024free},  allows to conclude the analysis of the smoothness of flat free boundaries for solutions to \eqref{eq:StefanNH}.

  The Stefan problem attracts the attention of many researchers since it models fundamental physical phenomena, such as the process of melting ice or freezing water, and it plays a pivotal role within the context of free boundary problems.
   
We recap some results concerning the viscosity formulation of the Stefan problem. For a more extensive discussion we refer to \cite{CS}. We point out the trilogy of articles \cite{ACS,ACS2, ACS3} that collect important contributions on the regularity in the homogeneous two-phase Stefan problem. In \cite{ACS} it is proved the optimal Lipschitz regularity across the free boundary and in \cite{ACS2,ACS3} the authors focus their attention on the so-called strong results, which involve obtaining an improvement of regularity starting from an initial one. The parabolic structure of the problem permits the construction of some counterexamples (see for instance \cite[Section 8.3]{CS}) for which Lipschitz free boundaries do not instantaneously regularize. However, if the Lipschitz constant is small enough or a nondegeneracy condition to prevent the heat fluxes to vanish on $\mathcal{F}(u)$ is prescribed, further regularity of the free boundary can be obtained. Moreover, if only a nondegeneracy condition is assumed, in \cite{ACS3}, the authors proved that it is possible to reach the same conclusion for sufficiently "flat" free boundaries, which do not even need to be graphs of functions. 
For results in the variable coefficients case see \cite{ferrari2010regularity,FeSa_p}. We also point out the interesting paper \cite{kim2024regularity} about the Hele-Shaw flow, which is a closely related topic.

In \cite{DFS23}, the authors adapted the viscosity elliptic approach introduced in \cite{DeSilva2011} to deal with the parabolic case. They addressed the regularity in a parabolic homogeneous one-phase scenario, employing the hodograph transform and viscosity theory. In particular, they introduced a new notion of a flat free boundary within a parabolic framework. 
   
   All of the aforementioned results on the Stefan problem consider only the homogeneous situation. However, the tools introduced in \cite{DFS23} are flexible enough to consider a source term, provided that suitable difficulties are addressed.  In this order of ideas, we describe the difficulties tackled in this paper, using the physical example of ice melting as a reference.
Recall that for the one-phase scenario, the condition on the free boundary, $\partial_t u=|\nabla u|^2$, only describes the melting of ice, since the time derivative is nonnegative, and so the interface between the two states is pushed over time towards the ice side.
Hence, the positive part of $f$ adds heat to the system following the natural direction of the boundary condition. On the other hand, a more delicate situation appears if we allow the source term $f$ to change sign, since this introduces a competition between the effect of the law of conservation of energy at the free boundary and the external source. 
The presence of a negative source term slows down the melting process of ice, potentially leading to a degenerate situation where no regularization is expected.
Indeed, this informal discussion translates into a very precise technical difficulty that appears when $f$ is allowed to be negative, which is the fact that the Hopf Lemma, as it is stated in the homogeneous setting, is false in this case. Hence, a stronger nondegeneracy condition has to be assumed.

Our goal falls within the scope of proving that flat free boundaries are smooth under minimal assumptions. As in \cite{DFS23}, the first step in this strategy is to characterize the geometry of the graph of $u$ near the $\varepsilon_0$-flat free boundary. We say that the free boundary is $\varepsilon_0$-flat if, at each time, within a small ball centered at a point on the free boundary, it is contained in a narrow strip (see Definition \ref{Def:flat_FB}).
The main result, conveniently formulated for the parabolic rescaling of \eqref{eq:StefanNH}, is the following

\begin{Theorem} \label{Theorem:Flatsoltoflatfree}
   Let $K>1$. There exists  $0 <\lambda \leq 1$ such that if $u$ is a viscosity solution of
    \begin{equation}\label{eq:StefanPara}
        \begin{cases}
      \partial_{t} u - F(D^2 u)  = \lambda f(\lambda x, \lambda^2 t)& \text{in} \,\, \left(B_2 \times (-(K\lambda)^{-1},0]  \right) \cap \{ u > 0  \},\\
       \partial_{t} u = \lambda \nor{\nabla u}^2& \text{on} \,\, \left(B_2 \times (-(K\lambda)^{-1},0]\right) \cap \partial\{u > 0  \},
    \end{cases}
    \end{equation}
    which satisfies
    \[
    0 \leq u \leq K, \quad  0\in \partial_x\{ u(\cdot,0)>0\} \text{ and } \; \partial_x \{u(\cdot,t)>0\} \textit{ is $\varepsilon_0$-flat in } B_2,
    \]
and for all $t_0 \ \in (-(K\lambda)^{-1},0)$ there exists $x_0 \in \p B_{\frac{3}{4}}$
    \begin{equation}  \label{eq:nondegscal}
         \left( \fint_{B_{r}(x_0) \times (t_0,t_0+K^{-1}r^2)} u^{p_0} \, dx \, dt\right)^{\frac{1}{p_0}} \geq K^{-1} (1+\norm{f^{-}}_{\infty}) \quad  \text{ for } r \leq \frac{1}{8}.
    \end{equation}
    Then for all $\eta >0$ we have
    \[
    \bar a(t)(x_n-b(t)-\eta^{1+\beta})^+\leq u(x,t) \leq \bar a(t)(x_n-b(t)+\eta^{1+\beta})^+
     \quad \mbox{in }B_{\eta} \times [-\lambda^{-1}\eta,0],\]
     where the coefficients satisfy
\begin{equation} \label{eq:esta}
    c(1+\|f^- \|_\infty) \le \bar a(t) \le C, \quad | \bar a'(t)| \le \eta^{\beta-2}, \quad \quad b'(t)=- \lambda \bar a(t), \quad b(0)=0,
\end{equation}
and the constants $p_0, \,\beta\in(0,1)$ and $c,C>0$ depend only on $K$ and $n$, provided that $\varepsilon_0$ is sufficiently small, depending on $\eta$, $n$ and $K$.
\end{Theorem}

 Theorem \ref{Theorem:Flatsoltoflatfree} states that, if
the free boundary of $u$ is $\varepsilon_0$-flat and $u$ satisfies \eqref{eq:nondegscal}, then the graph of $u$ near the free boundary is, at each time, trapped between two parallel hyperplanes that are very close together, see Figure \ref{Fig1}. This is a more useful property that plays a key role in obtaining higher regularity for the free boundary. Indeed from Theorem \ref{Theorem:Flatsoltoflatfree} we can directly invoke the machinery built in the recent work \cite{FFGJ} for the linear case, as well as \cite{wang2024free} for the fully nonlinear one. 
In particular, Theorem \ref{Theorem:Flatsoltoflatfree} for the case $f\equiv 0$ also justifies the definition of flatness considered in \cite{wang2024free}.

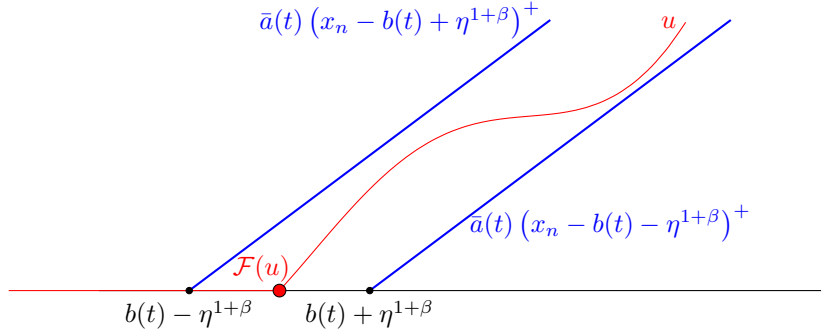
\begin{figure}[htbp]
    \centering
\begin{tikzpicture}[scale=1.2]
    \draw[-] (-1,0) -- (7,0);

    \draw[thick, blue] (0,0) -- (4,3) node[left] {$\bar a(t)\left(x_n -b(t) + \eta^{1+\beta} \right)^+$};
    \draw[thick, blue] (2,0) -- (6,3) node[near start, above, right] {$\bar a(t)\left(x_n -b(t) - \eta^{1+\beta} \right)^+$};
    
    \draw[scale=1, domain=-2:1, smooth, variable=\x, red] plot ({\x}, {0});
\draw[scale=1, domain=1:5.5, smooth, variable=\x, red] plot ({\x}, {sin(\x r -1 r)*1.1  + 0.2*(\x-1)*(\x-1)}) node[left] {$u$};
\filldraw (0,0) circle (1pt) node[below] {$b(t) - \eta^{1+\beta}$};
\filldraw[fill=red] (1,0) circle (2pt);
\node at (0.8,0.25) {\color{red} $\mathcal{F}(u)$};
    \filldraw (2,0) circle (1pt) node[below] {$b(t) + \eta^{1+\beta}$};
\end{tikzpicture}
\caption{The flatness condition expressed in  Theorem \ref{Theorem:Flatsoltoflatfree}}   
\label{Fig1}
\end{figure}

We now comment on the nondegeneracy condition \eqref{eq:nondegscal}. First, we emphasize that it depends on the $L^{\infty}-$norm of $f^{-}$. This fact appears to be sensible, since the model describes the melting of ice and the free boundary condition effect needs to prevail with respect to the contribution of the negative part of the source term. In fact, this ensures that close to the free boundary, the solution is \say{large enough to balance the effect of $f^{-}$}. This assumption prevents the water from freezing in the presence of a negative source term, which would conflict with the condition at the free boundary.

Furthermore, we remark that if $u$ is Lipschitz continuous with sufficiently small Lipschitz norm, the integral condition \eqref{eq:nondegscal} is equivalent to the more standad pointwise one 
\begin{equation*}
    u(x_0,t_0) \geq K^{-1}(1+\norm{f^-}_{\infty}).
\end{equation*}
Instead of assuming the Lipschitz continuity of $u$, we can replace \eqref{eq:nondegscal} with the alternative assumption 
\begin{equation} \label{eq:nondegpointf}
    u(x_0,t_0) \geq K^{-1}(1+\norm{f}_{\infty}),
\end{equation}
however it does not capture the difficulty associated with a negative source term, as discussed before. Notice that \eqref{eq:nondegpointf} reduces in the homogeneous case to the nondegeneracy condition considered in \cite{DFS23}. For a detailed discussion of this point, we refer to Section \ref{Section:Nondeg}.

A key step in proving Theorem \ref{Theorem:Flatsoltoflatfree} is the following refined version of the quantitative Hopf-Oleinik lemma obtained in \cite{caffarelli2013some} that we generalize for parabolic fully nonlinear equations. 
\begin{Proposition} \label{Theorem:CLN}
   Let $K>1$. For $0<T_1<T$ and $0<\delta<1$, there exist constants $\kappa, \mu>0$ which depend only on $n,\,\delta,\,T_1,\,T, \, K$ such that if $u$ is a viscosity solution of
    \begin{equation}\label{eq:cln}
    \begin{aligned}
        \begin{cases}
            \partial_t u- \mathcal M_K^- (D^2 u ) 
            \geq -\kappa  &\mbox{ in } B_1\times (0,T]\\
        u(x,0)\geq 1 &\mbox{ for } |x|\leq \delta\\
        u\geq 0 &\mbox{ in } \partial_p(B_1\times (0,T]).
        \end{cases}
    \end{aligned}
    \end{equation}
    Then
    \begin{equation} \label{eq:cln1}
            u(x,t)\geq \mu \, \dist(x,\p B_1)\quad \mbox{ in } B_1\times[T_1,T].
    \end{equation}
\end{Proposition}

Now we explain how Theorem \ref{Theorem:Flatsoltoflatfree}  fits within the literature. Indeed Theorem \ref{Theorem:Flatsoltoflatfree} establishes that, after rescaling, if the graph of $u$ is $\eta^{\beta}$-flat (in the sense of Definition \ref{Def:flat_FB}), then by choosing $\eta$ sufficiently small, we can ensure that it falls within the framework of \cite[Theorem 1.1]{FFGJ}, which we state below.

\begin{theoremA}
    \label{Theorem:flatsolution}
Let $K>1$. There exists  $0 <\lambda \leq 1$ such that if $u$ is a viscosity solution of the one-phase Stefan problem \eqref{eq:StefanNH} in $B_{2\lambda} \times [-2K^{-1}\lambda,0]$ when $F$ is the trace operator, assume that $0\in \p \{u>0\},$ and
 $$ \bar a(t) \, \left (x_n - b(t)-\varepsilon_1 \lambda \right)^+  \le u(x,t)  \le \bar a(t) \, \left(x_n - b(t)+\varepsilon_1 \lambda \right)^+,$$
 with
 $$c(1+\|f^- \|_\infty) \le \bar a \le C, \quad \quad |\bar a'(t)| \le \lambda^{-2}, \quad \quad b'(t)=-\bar a(t),\quad\quad  \|f\|_{L^\infty}\leq \varepsilon_1^2,$$
for some small $\varepsilon_1$ depending only on $K$ and $n$. Then in $B_\lambda \times[-K^{-1}\lambda ,0]$ the free boundary $\p \{u>0\}$ is a $C^{1,\alpha}$-graph in the $x_n$ direction. 
\end{theoremA}

Our main result Theorem \ref{Theorem:Flatsoltoflatfree} combined with Theorem A implies the $C^{1,\alpha}$ regularity of flat free boundaries when dealing with the heat operator.

\begin{Corollary}\label{Corollary:FlatFreeBoundary}
   Let $K>1$. There exists  $0 <\lambda \leq 1$ such that if $u$ is a viscosity solution of the one-phase Stefan problem \eqref{eq:StefanNH} in $B_{2\lambda} \times [-2K^{-1}\lambda,0]$ when $F$ is the trace operator, assume that $ 0\leq u \leq K \lambda$ and for all $t_0 \in (-K^{-1}\lambda,0)$ there exists $x_0 \in \p B_{\frac{3}{4}\lambda}$ 
    \begin{equation} \label{Nondeg_int}
       \left( \fint_{B_{r}(x_0) \times (t_0,t_0+K^{-1}r^2)} u^{p_0} \, dx \, dt\right)^{\frac{1}{p_0}} \geq K^{-1} (1+\norm{f^{-}}_{\infty}) \, \lambda \quad  \text{ for } r  \leq \frac{\lambda}{8},
    \end{equation}
     where $p_0 \in (0,1)$ is a universal constant.
    There exists $\varepsilon_0$ depending only on $K$ and $n$ such that if, for each $t$, $\p_x \{ u(\cdot, t) >0 \}$ is $\varepsilon_0-$flat in $B_{\lambda}$, then the free boundary $\p\{ u >0 \}$ (and $u$ up to the free boundary) is $C^{1,\alpha}$ in $B_{\frac{\lambda}{2}} \times [-(2K)^{-1} \lambda, 0]$.
\end{Corollary}

Note that the assumption $u\leq K\lambda$ in a set of size $\lambda$ is natural since this always holds for classical solutions, provided we choose $\lambda$ sufficiently small, see \cite{DFS23}.

Now we conclude describing the structure of the paper. In Section \ref{Section:preliminaries} we provide the necessary notation, the main definitions, and we include the proof of Proposition \ref{Theorem:CLN}. The proof of Theorem \ref{Theorem:Flatsoltoflatfree} is presented in Section \ref{Section:flatfbtoflatsol}. Then we conclude with Section \ref{Section:Nondeg} commenting the integral nondegeneracy condition \eqref{Nondeg_int} and its relation with the pointwise condition considered in \cite{DFS23}.



    \section{Preliminaries}\label{Section:preliminaries}
    \subsection{Notation and definitions} \label{Subsection:notation}
Let $(x,t)\in \R^{n+1}$, we call $x\in \mathbb{R}^n$ the space variable and $t \in \R$  the time variable. A point $x\in \mathbb{R}^n$ will sometimes be written as $x=(x',x_n)$, where $x'\in \R^{n-1}$ and $x_n \in \R$.
Given a function $u(x,t)$, we denote by $\nabla u$ the gradient of $u$ with respect to $x$, $D^2 u$ the Hessian with respect to $x$, $\nabla_{x'} u$ the gradient with respect to $x'$, $\partial_{x_n}u $  and $\partial_{t} u$ the partial derivatives of $u$ with respect to $x_n$ and $t$, respectively. 

A \textit{parabolic cylinder} $\mathcal{P}$ is any set of the form 
 $$\mathcal{P}=U \times (t_1,t_2],$$ 
 where $U$ is a smooth bounded domain in $\R^n$ and $t_1 < t_2$. We denote by $\partial_p \mathcal{P}$ the \textit{parabolic boundary} of $\mathcal{P}$, i.e., $\partial_p \mathcal{P}=\left( U\times \{t_1\}  \right)  \cup   \left( \partial U \times [t_1,t_2] \right) $.
 In particular, we define the \textit{standard parabolic cylinder} of radius $r$ as

 \[ 
 \mathcal{P}_r(x_0,t_0) = B_r(x_0) \times (t_0,t_0+r^2]
 \]
 We will omit the center if it is the origin.
Given $(x,t), (y,s)\in \R^{n+1}$, we define the \textit{parabolic distance} as
$$
d_p((x,t), (y,s)) :=\big(|x-y|^2 + |t-s|\big)^{1/2}.
$$
Let $C^{0}(\mathcal{P})$ be the set of continuous function in $\mathcal{P}$. For $\alpha \in (0,1]$ and $u\in C^{0}(\mathcal{P})$, define the parabolic H\"older semi-norm as
\begin{equation*}
    [u]_{C^{\alpha}(\mathcal{P})} := \sup_{\substack{(x,t)\neq (y,s)} \in \mathcal{P}} \frac{\nor{u(x,t)-u(y,s)}}{d_p((x,t),(y,s))^{\alpha}}.
\end{equation*}
The parabolic H\"older norm is defined as 
\begin{equation*}
    \norm{u}_{C^{\alpha}(\mathcal{P})} := \norm{u}_{L^{\infty}(\mathcal{P})} +[u]_{C^{\alpha}(\mathcal{P})}.
\end{equation*}
Further we define the H\"older semi-norm in time to be
\begin{equation*}
[u]_{C_t^{\alpha}(\mathcal{P})} := \sup_{ \substack{(x,t_1)\neq (x,t_2)}\in \mathcal{P}} \frac{|u(x,t_1)-u(x,t_2)|}{|t_1-t_2|^\alpha}.
\end{equation*}
This is necessary to introduce  the $C^{1,\alpha}(\mathcal{P})$-norm as
\begin{equation*}
    \norm{u}_{C^{1,\alpha}(\mathcal{P})}:=\norm{u}_{L^{\infty}(\mathcal{P})} + \norm{\nabla u}_{L^{\infty}(\mathcal{P})} +[\nabla u]_{C^{\alpha}(\mathcal{P})} + [u]_{C_t^{(1+\alpha)/2}(\mathcal{P})}.
\end{equation*}
We denote by $C^{\alpha}(\mathcal{P})$ (resp. $C^{1,\alpha}(\mathcal{P})$) the space of continuous (resp. differentiable in space) functions on $\mathcal{P}$ which are bounded in the $\norm{\cdot}_{C^{\alpha}(\mathcal{P})} $ (resp. $\norm{\cdot}_{C^{1,\alpha}(\mathcal{P})}$) norm. 
Throughout the paper, we will denote by K a large universal constant.
The operator $F:\mathcal{S}(d)\to\mathbb{R}$ is uniformly $K$-elliptic if it satisfies
\[
	K^{-1} |N|\leq F(M)-F(M+N) \leq K|N|,
\]
for every $M,N\in\mathcal{S}(d)$, with $N\geq 0$.
Recall the definition of extremal Pucci operators with ellipticity constants $K^{-1}$ and $K$  
\begin{equation}\label{Pucci} \mathcal M_K^+ (N) = \sup_{K^{-1} I \le A \le K I} \quad \tr (A N), \quad \quad  \quad \mathcal M_K^- (N) = \inf_{K^{-1} I \le A \le K I} \quad \tr ( A N).\end{equation} 
Let $f$ be a continuous function, we denote by $\underline{S}_K(f)$ the space of upper semi-continuous functions $u$ in $\mathcal{P}$ such that $\p_t u - M_K^+(D^2 u) \leq f$ in $\mathcal{P}$. The classes $\overline{S}_K(f), S_K(f)$ and $S_K^{*}(f)$ are defined as usual.
In the rest of the paper the constant $\alpha_0$ always refers to the universal exponent corresponding to the  $C^{1,\alpha_0}$ regularity  of solutions to the problem $\p_t u - F(D^2u)=0$ where $F$ is any $K$-elliptic operator, see for instance \cite{wang1992regularityII}. \\
In what follows $c,\, C, \, C_1, \dots$ denote universal constants that may vary from line to line. 

We now define what we mean by a viscosity solution of \eqref{eq:StefanNH}.
    \begin{Definition}\label{Def:viscosity_stefan}
     We say that $u\in C(\Omega\times [0,T])$ is a viscosity solution of \eqref{eq:StefanNH} if its graph cannot be touched by above (resp. below), at a point $(x_0,t_0)$ in a parabolic cylinder $B_r (x_0) \times (t_0 - r^2,t_0],$
      by the graph of a classical strict supersolution $\varphi^+$ (resp. subsolution). 
      By a classical strict supersolution we mean that $\varphi(x,t) \in C^{2}$, $\nabla \varphi \neq 0$ and it solves
    \begin{equation}\label{eq:StefanNH1}
        \begin{cases}
       \partial_{t} \varphi -F(D^2 \varphi) > f& \text{in} \,\, \left(\Omega \times (0,T] \right) \cap \{ \varphi > 0  \},\\
       \partial_{t} \varphi > \nor{\nabla \varphi}^2& \text{on} \,\, \left(\Omega \times (0,T] \right) \cap \partial\{ \varphi > 0  \}.
    \end{cases}
    \end{equation}
    Similarly we can define a strict classical subsolution.
    \end{Definition}

The definition of $\varepsilon_0-$flatness of the free boundary is given hereafter.
\begin{Definition}\label{Def:flat_FB}
    Let $\partial_x\{u(\cdot,t)>0\}$ denote the spatial boundary in $\R^n$ of $\{u(\cdot,t)>0\}$. We say that $\partial_x\{u>0\}$ is $\varepsilon_0$-flat in $B_\lambda(x_0)$ if, for every time $t$, there exists a direction $\nu$ such that
    \[
    \partial_x\{u(\cdot, t)>0\}\cap B_\lambda(x_0)\subset \{|\ps{x-x_0}{\nu}|\leq \varepsilon_0\lambda\}
    \]
    and
    \begin{align*}
        u=0 \quad &\textit{ in } \quad \{\ps{x-x_0}{\nu}\leq -\varepsilon_0\lambda\},\\
        u>0 \quad &\textit{ in } \quad \{\ps{x-x_0}{\nu}\geq \varepsilon_0\lambda\}.
    \end{align*}
\end{Definition}

\subsection{Parabolic Hopf-Oleinik lemma}
We devote this subsection to prove Proposition \ref{Theorem:CLN} which will be instrumental to prove Theorem \ref{Theorem:Flatsoltoflatfree}. We start by recalling the Weak Harnack inequality for fully nonlinear parabolic equations  \cite[Corollary 4.14]{wang1992regularityI}.
    \begin{Proposition}[Parabolic Weak Harnack inequality] \label{WeakHarnackIneq}
        Let $f \in L^{\infty}(\mathcal{P}_1) \cap C(\mathcal{P}_1)$ and $u$ be a nonnegative solution of $\partial_t u- F(D^2 u) \geq f \mbox{ in } \mathcal{P}_1.$
        Then for every $r<\frac{1}{4}$ it holds
        \[
        \left(\fint_{B_r\times(r^2, 2r^2)} u^{p_0}\right)^\frac{1}{p_0}\leq C_H\left(\inf_{B_r\times(3r^2, 4r^2)} u + r^2\|f\|_{L^\infty(\mathcal{P}_1)}\right)
        \]
        where $p_0 \in (0,1),\, C_H>0$ depend only on $n, \, K$.
    \end{Proposition}
Now we are in position to prove the Hopf-Oleinik Lemma for viscosity solutions of fully nonlinear parabolic equations, as stated in Proposition \ref{Theorem:CLN}. The proof follows the arguments presented in \cite{caffarelli2013some}.
\noindent\textit{Proof of Proposition~\ref{Theorem:CLN}.~}
Consider the barrier function $h$ as
$$
h(x,t):=\frac 1 D(E-F),
\ \ 
E= (t+a)^{-b} 
e^{ -\frac {K |x|^2}{ t+a}  },
\ \
F=(T+a)^{-b}
e^{ -\frac {K}{ T+a}  }, \ \ D= a^{-b}-F ,
$$
with
$$
\frac {T+a}{ 4a}= \frac 2{ \delta^2},\quad b:=
\frac K {  (T+a) \log(1+\frac Ta) }
=\frac{
 K (8-\delta^2) }
{ 8  \log( 8\delta^{-2}) } \frac{1}{T}. 
$$
Note that $h(x,T)=0 $ for $|x|= 1$, hence along the sides of $B_1 \times (0,T)$  the function $h$ is nonnegative being increasing in time, indeed
\begin{align*}
    \p_t h = \frac{E}{D}\frac{1}{(t+a)^2} \left[ K \nor{x}^2 - b(t+a )\right],
\end{align*} 
and  $b(T+a) \leq \frac{K}{\log(8\delta^{-2})} < K$ implies that 
\[
 \p_t h (x,t) =  \frac{E}{D}\frac{1}{(t+a)^2} \left[ K - b(t+a )\right] >0.
\]
Defined in this way we also have that
\begin{align*}
    h(x,0) \leq h(\frac{\delta}{2},0) = 0, \quad \text{for }|x| \geq \frac{\delta}{2} \qquad  h(x,0) \leq h(0,0) =1 \quad \text{for all }x. 
\end{align*}
Thus $h \leq u$ on $\partial_p(B_1\times (0,T])$. We need to prove that exists a  $\Tilde{T} > 0$ such that for every $T \leq \Tilde{T}$ we have 
\begin{equation} \label{eq:obj}
     \partial_t h - \mathcal M_K^- (D^2 h )  < -\kappa  \mbox{ in } B_1 \times (0,T], 
\end{equation} 
where $\kappa$ is some constant depending only on $n,K,\delta$ and $T$.
Computing the spatial derivatives of $h$
\begin{equation} 
    D^2 h = \frac{E}{D}\frac{2K}{(t+a)^2} \left[ 2K (x \otimes x) - (t+a) I_n  \right].
\end{equation}
By the property of the Pucci extremal operators \cite[Lemma 2.10 (3)-(5)]{cc} we have that
\begin{align*}
    \mathcal M_K^- (D^2 h ) \geq  \frac{E}{D}\frac{2K^2}{(t+a)^2} \left[ 2 \mathcal   M_K^-\left( x \otimes x \right) - (t+a) n  \right] = \frac{E}{D}\frac{2K}{(t+a)^2} \left[  2 \nor{x}^2 -  nK (t+a)   \right].
\end{align*}
Hence
\begin{equation} \label{eq:cln2}
\begin{aligned}
    \p_t h - \mathcal M_K^- (D^2 h )  \leq   \frac{E}{D}\frac{1}{(t+a)^2} \left[ K \nor{x}^2 - b(t+a) - 4 K \nor{x}^2 + 2nK^2 (t+a)\right]  \\ 
    = \frac{E}{D}\frac{1}{(t+a)^2} \left[ -3 K \nor{x}^2 + (t+a)(2nK^2 - b)\right] < - \kappa.
\end{aligned}
\end{equation}

Choosing $\Tilde{T}$ small enough depending only on $n,K,\delta$ we have that $2nK^2 - b  < 0$ which gives \eqref{eq:obj}. In order to recover the general argument for $T > 0$ one just needs to iterate the reasoning. Indeed, one can fix  a positive integer
$m$ so that $\frac Tm<\Tilde T$, and then apply
the result on $[0, \frac Tm]$, $[\frac Tm,
\frac {2T}m]$, ..., $[\frac {(m-1)T}{m}, T]$ successively.

Consider $T \in (0,\Tilde{T})$, in the following we show that $u \geq h $ in $B_1 \times (0,T]$ which implies $u(x,T) \geq \mu (1-|x|)$ for $x \in B_1$ where $\mu$ is some positive constant depending only on $n,K,\delta$ and $T$. In particular this gives \eqref{eq:cln1} since $\kappa$ and $\mu$ can be chosen to depend on $T$ monotonically. 
In order to achieve our goal consider a translation of the graph of $h$ in such a way that stays below $u$ in $B_1 \times (0,T]$. In particular for a constant $c_0$ let $h-c_0$ so that its graph touches that of $u$ from below at some point $(\bar x, \bar t) $. It suffices to prove that $c_0 \leq 0$. If $c_0 >0$ then $ (h-c_0) < u$ on $\partial_p(B_1\times (0,T])$ from the costruction of $h$. Hence $(\bar x, \bar t) \in B_1 \times [0,T]$ and since $u \in \overline{S}_K(-\kappa)$
we have at $(\bar x, \bar t)$
\[
\partial_t (h-c_0)- \mathcal M_K^- (D^2 (h-c_0) ) 
            \geq -\kappa
\]
which gives a contradiction with \eqref{eq:cln2}. 
\qed
\section{From flat free boundaries to flat solutions} \label{Section:flatfbtoflatsol}

The following lemma localizes the free boundary in time. This quantifies the increase in time of the support which is implied by the free boundary condition and the definition of flatness. 
\begin{Lemma} \label{Lemma:locationFB}
Let  $u$ be a solution of \eqref{eq:StefanPara}  in $B_2 \times (-K^{-1},1]$. Assume $0 \leq u \leq K$ and $\partial_x \{u(\cdot,t)>0\} \textit{ is $\varepsilon_0$-flat in } B_2$.
If $u(x,0)=0$ in $B_1$, then 
\begin{equation} \label{eq:2.1}
    u(x,t) \leq C\left( \nor{x} -1 \right)^{+}, \quad \text{in } B_2 \times (- (2K)^{-1},0] ,
\end{equation}
and
\begin{equation} \label{eq:2.2}
    u(x,t) = 0 \quad  \text{for } \nor{x} < 1 - C\lambda, \, t \in [0,1],
\end{equation}
with $C > 0$ universal.
\end{Lemma}
\begin{proof}
    From the free boundary condition and the flatness assumption it follows immediately that the support of $u$ is increasing in time, hence $u=0$ in $B_1$ for $t \in [-K^{-1},0]$. Consider $v$ a solution of
 \begin{equation}\label{eq:CompareFunction}
        \begin{cases}
      \p_t v - F(D^2v) = \lambda f & \text{in} \,\, \left(B_2 \setminus B_1 \right) \times (-K^{-1},0]  ,\\
      v= 0 & \text{on} \,\,  \partial B_1 \times [-K^{-1},0] , \\
      v= K & \text{on}  \,\, \left( \partial B_2 \times [-K^{-1},0] \right) \cup \left( \left( B_2 \setminus B_1 \right) \times \{ -K^{-1} \}\right),
    \end{cases}
    \end{equation}
then $u-v \leq 0 $ on $\partial_{p}\left( (B_2 \setminus B_1) \times (-K^{-1},0]\right)$ and since $u - v \in S_K$ applying maximum principle allows to conclude that $u \leq v$ in $(B_2 \setminus B_1) \times  [-K^{-1},0]$.
For the regularity up to the boundary given by \cite[Theorem 2.1]{wang1992regularityII} and if we
let $\Bar{x} = \frac{x}{\nor{x}} \in \partial B_1$, we can see that
\begin{align*}
    u(x,t)  \leq v(x,t) = v(x,t) - v(\Bar{x},t) \leq \left( \max_{ \overline{B_2} \setminus B_1} \nor{\nabla v
    (\cdot, t)}\right) \nor{x-\Bar{x}} \leq C \nor{x-\frac{x}{\nor{x}}} = C(\nor{x}-1).
\end{align*}
If we extend the comparison in $B_1 \times [-K^{-1},0]$ the estimate \eqref{eq:2.1} follows. \\
For $t \in [0,1]$ we compare $u$ with 
\begin{equation*}
    w(x,t) = C_{0}\, g( \nor{x} - r(t)), \quad r(t):= (1-C_0 \lambda t),
\end{equation*}
with $g$ such that $g(s) = 0 $ for $s \leq 0$, and for $s$ positive it satisfies the ODE 
\begin{equation*}
\begin{cases}
    g''(s) + 2ng'(s)= 0 \text{ for } s > 0, \\
    g(0) = 0, \quad g'(0) = 1.
\end{cases}
\end{equation*}
Note that it is possible to compute explicitly $g$ as
\begin{equation*}
    g(s) = \frac{1-e^{-2ns}}{2n} \chi_{[0,\infty[}(s),
\end{equation*}
 for $(x,t) \in B_2 \times [0,1]$ we have that $\nor{x} - r(t) \leq 1+C_0 \lambda$  implies that $ g' \in [e^{-2n(1+C_0 \lambda)},1]$.
We may assume that $r(t) \geq 1/2$, otherwise if $C > 2 C_0$ then $1-C\lambda <0$ and the conclusion \eqref{eq:2.2} is trivial. From this we get that 
\begin{equation*}
    g(\nor{x}- r(t)) \geq 0 \Leftrightarrow  \nor{x} \geq r(t) \Rightarrow \nor{x} \geq \frac{1}{2}.
\end{equation*}
Exploiting the fact that $f$ is bounded in $B_2 \times [0,1]$, and assuming
ùthat $C_0$ is chosen such that $w \geq u$ at time $t=0$, from \eqref{eq:2.1},  also holds on $\p B_2 \times [0,1]$ and is a supersolution of \eqref{eq:StefanPara}, in fact
\begin{equation*}
\begin{aligned}
     \partial_{t} w - \Mm(D^2 w)& = \p_t w - K^{-1}\tr( [D^2 w]^+) + K \tr( [D^2 w]^-) \\
    & =C_0^2\lambda g' - C_0 \left[ K^{-1} (n-1) \frac{g'}{|x|} - K \left(g'' - \frac{g'}{|x|}\right)\right] \\&= C_0g' \left(C_0  \lambda - K^{-1}\frac{n-1}{|x|}- 2nK -K \frac{1}{|x|}\right) \\ 
    &\geq C_0g' (C_0 \lambda - 4nK) >f
    \quad \text{in } \{w >0\}, \\
     \partial_t w &= \lambda C_0^2 = \lambda \nor{\nabla w}^{2} \quad \text{on } \partial\{w >0 \} .
\end{aligned}
\end{equation*}
Hence $w$ cannot touch by above $u$ both in $\{w > 0 \}$ and in $\partial \{ u >0 \}$ which means that $u \leq w \quad \text{in } B_2 \times [0,1]$ and leads to \eqref{eq:2.2}. 
\end{proof}

We are now ready to provide the proof of Theorem \ref{Theorem:Flatsoltoflatfree}. \\ 

\noindent\textit{Proof of Theorem~\ref{Theorem:Flatsoltoflatfree}.~}
The proof of this result follows the strategy of \cite{DFS23}. The presence of a possibly negative source term introduces new difficulties. The argument revolves around trapping the solution between two parallel affine functions at each time: from below with the Hopf-Oleinik Lemma in Proposition \ref{Theorem:CLN} and from above exploiting up to the boundary regularity of solutions of the homogeneous equation. Then we perturb the family of evolving planes into travelling waves.\\ 
Since $(0,0) \in \p \{ u > 0\}$, from the $\varepsilon_0-$flatness of $\p_x \{ u > 0\} $, possibly after a rotation, we can reduce to the situation in which
\[  u(x,0) > 0 \text{ in } \{ x_n > \varepsilon_0 \}\cap B_2,  \quad  u(x,0) = 0 \text{ in } \{ x_n < - \varepsilon_0 \}\cap B_2.
\]
For any $x \in B_{1/2} \cap \{x_n > - \varepsilon_0  \}$, consider a ball $B_r(x_0)$ tangent from below to $ \{x_n = -\varepsilon_0 \}$ 
with $x_0=(x',-\varepsilon-r)$ with $r$ such that $B_{2r}(x_0) \subset B_{\frac{3}{4}}$. In this way rescaling \eqref{eq:2.1} gives
\begin{equation} \label{eq:ineq00}
    u(x,t) \leq C(|x-x_0| -r)^+ = C(x_n + \varepsilon_0)^+ \quad \text{in } B_{1/3} \times [-(3K)^{-1},0]
\end{equation}
Now consider a parabolic rescaling of $u$, namely $u_\tau(x,t) := \frac{1}{\tau}u(\tau x, \tau^2 t)$, applying \eqref{eq:ineq00} we have
\begin{equation} \label{eq:ineq1}
    u_\tau \leq C (x_n + \frac{ \varepsilon_0}{ \tau})^+ \leq  C (x_n + \tau)^+ \text{ in } B_{\frac{1}{3\tau}} \times [-(3K\tau^2)^{-1},0]
\end{equation}
for some $\tau \geq \varepsilon_{0}^{1/2}$. So for $\tau \in (\varepsilon_0^{1/2},c)$ we have $u_\tau \leq C(x_n +\tau)^+$ in $B_1 \times (-2,0)$ where $c$ is universal.
Note further that $u_\tau $ satisfies \eqref{eq:StefanPara} with $\tau \lambda$ instead of $\lambda$. Applying \eqref{eq:2.2} to $u_{\tau}$ we obtain that \begin{equation} 
    u_\tau = 0 \text{ in } \{x_n \leq -\tau - C\lambda \tau  \} \times [-1,0].
\end{equation} 
 And since $(0,0) \in \p \{ u >0 \}$ we have
 \begin{equation}\label{eq:flat1}
     (\p_x \{ u_\tau(\cdot, t) > 0\} \cap B_{1/2}) \cap \{ x_n \leq C\lambda \tau \} \neq \emptyset \quad \text{ for all }t \in [-1,0].
 \end{equation} 
Moreover, $\p_x \{ u_\tau > 0\}$ is $\tau^{-1} \varepsilon_0$-flat in $B_1$, which together with \eqref{eq:flat1} gives
 \begin{equation} \label{eq:flat2}
      \p \{ u_\tau > 0\} \cap (B_{1/2} \times [-1,0]) \subset \{ x_n < C(\varepsilon_0 \tau^{-2}+ \lambda) \tau\} .
 \end{equation}
Now since $\tau \geq \varepsilon_{0}^{1/2}$, we compare $u_\tau$ in $(B_{1/2} \cap \{x_n > C(1+\lambda)\tau \} )\times [-1,0] =\Omega$ with $w$ which satisfies
 \begin{equation}\label{eq:CompareFunction1}
        \begin{cases}
\p_t w - F(D^2 w) = 0 & \text{in} \,\, \Omega ,\\
      w= 0 & \text{on} \,\, (B_{1/2} \cap \{x_n = C(1+ \lambda) \tau  \}) \times [-1,0] = A_\tau  \times [-1,0], \\
      w= u_\tau & \text{on}  \,\, \p_p \Omega \setminus (A_\tau \times [-1,0]).
    \end{cases}
    \end{equation}
    Note that by \eqref{eq:ineq1} we have that $|u_\tau - w| \leq C((1+\lambda)C\tau + \tau)\leq C_2 \tau $ on  $ \p_p \Omega$  on  $ \p_p \Omega$. Thus, by maximum principle \cite[Cor. 3.21]{wang1992regularityI}, since  $u_\tau - w \in \underline{S}_K(\lambda \tau f)$,  we have
    \begin{equation} \label{eq:est22}
    \sup_{\Omega} |u_\tau - w| \leq \sup_{\p_p \Omega} |u_\tau - w| + C\lambda \tau \|f^- \|_{\infty} \leq C_2(1+ \lambda \norm{f^-}_{\infty}) \tau.
\end{equation}
 Furthermore, the $C^{1,\alpha_0}$ regularity up to the boundary of $w$ on the flat boundary $(B_{1/2} \cap \{ x_n = C(1+\lambda)\tau\}) \times [-1,0]$ (see \cite[Theorem 2.1]{wang1992regularityII}) implies that there exists a constant $\bar a$, $|\bar a| \leq C_1$ with $C_1$ universal, such that for any $\rho$ small
\begin{equation} \label{eq:est23}
    |w - \bar a x_n| \leq C_3 \rho^{1+\alpha_0} \quad \text{in } B_{\rho}^+ \times [-\rho^2,0]
\end{equation}
 where $ B_{\rho}^+:= (B_\rho(0,C(1+\lambda)\tau) \cap \{  x_n > C(1+\lambda)\tau \})$. Thus \eqref{eq:est23} combined with \eqref{eq:est22} gives
\begin{equation}\label{ustau}
|u_\tau - \bar a x_n| \le C_3 \rho^{1+\alpha_0} + C_2(1+ \lambda \norm{f^-}_{\infty}) \tau \le C(1+\lambda \norm{f^-}_{\infty}) \rho^{1+\alpha_0} \quad \quad \mbox {in} \quad B^+_{\rho} \times[-\rho^2,0],
\end{equation}
 provided that we choose $\tau = \rho^{1+\alpha_0}$ with $\rho$ small, to be made precise later. Now we make the following claim.

 \noindent \textbf{Claim}: The nondegeneracy condition \eqref{eq:nondegscal} implies that $\bar a>c(1+\|f^- \|_{\infty})$, where $c>0$ is universal.

\textit{Proof of the claim:} From \eqref{eq:flat2} it follows that 
\[
u_\tau >0 \text{ in } (B_{1/2} \cap \{x_n>C(\varepsilon_0 \tau^{-2} + \lambda)\tau \}) \times [-1,0]
\]
which written in terms of $u$ reads as
\begin{equation} \label{eq:loc1}
    u >0 \text{ in } (B_{\tau/2} \cap \{x_n>C(\varepsilon_0  + \lambda\tau^2) \}) \times [-\tau^2,0].
\end{equation}
In particular, \eqref{eq:loc1} ensures that $\p_x \{ u( \cdot, t) >0 \} $ lies below $\{x_n = C (\varepsilon_0 + \lambda |t|) \}$ at all times $t= -\tau^2 \leq -\varepsilon_0$. 
Since \eqref{eq:nondegscal} holds, we can assume $x_0 \in \p B_{3/4} \cap \{x_n  >  C (\varepsilon_0 + \lambda |t|)  \}$ which is the worst case of $x_0$ in the farthest position away from the free boundary.  Now, defining $R:= d(x_0,\p_x \{u(\cdot, - (K\lambda)^{-1}) >0 \})$, since the support of $u$ increases with time, then $u$ is a viscosity solution of $\p_t u - F(D^2 u ) = \lambda f$ in  $B_R(x_0) \times [- (K\lambda)^{-1},0]$. Moreover combining weak Harnack inequality of Proposition \ref{WeakHarnackIneq}, applied with $-\lambda f^{-}$ as the right hand side, with the nondegeneracy \eqref{eq:nondegscal} for any $r \leq \min\{\frac{1}{8},\frac{R}{2}, (2KC_H)^{-1/2}\}$ we get 
\begin{align} \label{eq:weak11}
    &\inf_{B_r(x_0)\times(t_0+K^{-1}r^2, t_0+2K^{-1}r^2)} u \geq \frac{1}{C_H}  \left(\fint_{B_r(x_0)\times(t_0,t_0+K^{-1}r^2)} u^{p_0}\right)^\frac{1}{p_0} - r^2 \lambda \|f^-\|_{\infty} 
     \\ &\geq \frac{(1+\|f^-\|_\infty)}{K C_H}   - r^2 \lambda \|f^-\|_{\infty} \geq  (2KC_H)^{-1} (1+\| f^-\|_\infty) \quad \text{for }t_0 \in (-(K\lambda)^{-1},0)
\end{align}
Since $\lambda \leq 1$, the particular choice $t_0 = -\frac{1}{2K}(1+3r^2) \in (-K^{-1},0)$ implies that 
$$\Tilde{T}:= \frac{1}{2K} \in (t_0+K^{-1}r^2, t_0+2K^{-1}r^2)= (-\frac{1}{2K}(1+r^2), -\frac{1}{2K}(1-r^2))\subset (-\frac{1}{K},0),$$
and thus
\begin{equation} \label{eq:defT}
    u(x,\Tilde{T}) \geq (2KC_H)^{-1} (1+\| f^-\|_\infty) \quad \textit{for }x \in B_r(x_0) 
\end{equation}
Define now the rescaled $\bar u= A u(Rx + x_0, R^2 t)$, $\bar T:=\frac{\Tilde{T}}{R^2} $ and $\bar F (M):= AR^2F(\frac{M}{AR^2})$, which is still uniformly $K$-elliptic.
Choosing conveniently 
\[
 A:=\frac{4KC_H}{R^2 (1+\| f^-\|_{\infty})}, \qquad \delta:=R^{-1}r
\]
the function $\bar u$  solves
\begin{align*}
        \begin{cases}
            \partial_t \bar u- \bar F(D^2 \bar u ) \geq-\kappa &\mbox{ in } B_1\times (\bar T,0]\\
        \bar u(x,\bar T)\geq 1 &\mbox{ for } |x|\leq \delta\\
        \bar u\geq 0 &\mbox{ in } \partial_p(B_1 \times (-(2R^2K\lambda)^{-1},0].
        \end{cases}
    \end{align*}
The first condition is guaranteed for $ \lambda \leq \frac{\kappa}{4KC_H}$ 
 \[
 -AR^2 \lambda \|f^-\|_\infty = -4K C_H \lambda \frac{\| f^-\|_{\infty}}{(1+\| f^-\|_{\infty})} \geq -4K C_H \lambda, 
\]
 and the second follows from \eqref{eq:defT}
 \[
\bar u(x,\bar T) = A u(R x+ x_0,\tilde T) \geq A (2KC_H)^{-1} (1+\| f^-\|_\infty) \geq \frac{2}{R^2} \geq 1.
 \]
The third condition follows directly from $u \geq 0$ in $\p_p (B_R(x_0) \times (-(2K\lambda)^{-1},0])$. We are now in position to apply Proposition \ref{Theorem:CLN} to $\bar u$ which gives for $T_1 \in (\bar T,0]$
\begin{align*}
    \bar u(x ,t) \geq \mu (1-|x|)  \quad \text{for }(x,t) \in B_1 \times (T_1,0).
\end{align*}
Rescaling back to the original function $u$ gives
\begin{equation}\label{eq:hopf1}
\begin{aligned}
    A u(Rx + x_0, R^2 t) &\geq (1-|x|) \quad \text{for }(x,t) \in B_1 \times (T_1,0) \\
    u(y,s) &\geq \frac{\mu}{A}\left(1-\left|\frac{y-x_0}{R}\right|\right) = \frac{\mu R(1+\| f^-\|_{\infty})}{4KC_H}d(y, \p B_R(x_0)) \\
    & \geq \frac{\mu}{8KC_H}(1+\| f^-\|_{\infty})d(y, \p B_R(x_0)) \quad \text{for }(y,s) \in B_R(x_0) \times (R^2 T_1,0).
\end{aligned}
\end{equation}
where in the second line we have replaced $y=Rx+x_0$ and $s=R^2t$. 
Rewriting \eqref{ustau} with $\beta=\frac{1}{1+\alpha_0}$ we reach
\begin{equation} \label{eq:hopf3}
\begin{aligned}
    \left|\frac{1}{\tau}u(\tau x,\tau^2 t) - \bar a x_n \right| &\leq C(1+\lambda\| f^-\|_\infty)\tau \\
    \left|u(\tau x,\tau^2 t) - \bar a \tau x_n \right| &\leq C(1+\lambda\| f^-\|_\infty)\tau^2 \quad \text{in }B_{\tau^{\beta}}^+ \times (-1,0) \\
     \left|u(y,s) - \bar a y_n \right| &\leq C(1+\lambda\| f^-\|_\infty)\tau^2 \quad \text{in }P_\tau 
\end{aligned}
\end{equation}
where we have made the substitution $\tau x=y, \tau^2t=s$ and $P_\tau$ is defined as $P_\tau:=(B_{\tau^{\beta+1}}(0,C(1+\lambda)\tau^2) \cap \{ y_n \geq C(1+\lambda)\tau^2\}) \times (-\tau^2,0)$. Now choosing $(y,s) \in P_\tau, \, y=(0,C(1+\lambda)\tau^2 + \frac{1}{2}\tau^{\beta+1})$ in \eqref{eq:hopf1} together with \eqref{eq:loc1} one get for $s=0$
\begin{equation} \label{eq:hopf2}
    u(y,0) \geq C(1+\|f^- \|_\infty)(y_n - C\varepsilon_0)^+.
\end{equation} 
Now combining \eqref{eq:hopf2} and \eqref{eq:hopf3} since we have $\varepsilon_0 \leq \tau^2$
\begin{align*}
    \bar a y_n &\geq C(1+\| f^-\|_{\infty})(y_n - C\varepsilon_0)^+ - C(1+\lambda\| f^-\|_\infty)\tau^2 \\
    \bar a (2 \tau^{\beta+1}) &\geq    C(1+\|f^- \|_\infty )(C\tau^{\beta+1}-C\tau^2) - C(1+\lambda\| f^-\|_\infty)\tau^2  \geq \frac{C}{2}(1+\|f^- \|_\infty ) \tau^{\beta+1} \\
    \bar a &\geq \frac{C}{4}(1+\|f^- \|_\infty ),
\end{align*}
which gives the desired claim.

The estimate\eqref{eq:hopf3} for $\gamma = \frac{\alpha_0}{2+\alpha_0} $ can be restated as
$$ (\bar a x_n - C (1+\lambda\| f^-\|_\infty)\eta^{1+ \gamma})^+ \le u \le  (\bar a x_n + C (1+\lambda\| f^-\|_\infty)\eta^{1+ \gamma})^+ \quad \mbox{in} \quad B_{2 \eta} \times [-\eta^2,0],$$
with $\eta=\tau^{\beta+1}$.

Now we proceed similarly, by looking at the points $(b(t_0)e_n,t_0)$ where the free boundary intersects the $x_n$ axis, for which we can exploit the information of \eqref{eq:loc1} and reach 
\begin{equation} \label{eq:b1}
    |b(t_0)|\le C (\lambda |t_0|+\varepsilon_0) \le C_0 \eta \quad \mbox{if} \quad t_0 \in [-\lambda^{-1} \eta,0].
\end{equation} 
Combining \eqref{eq:loc1} and \eqref{eq:b1} we have that for $t \in (t_0-\eta^2, t_0)$
\begin{align*}
    \p_x \{x-b(t_0)e_n : u(x,t)=0 \} &= \p_x \{x : u(x+b(t_0)e_n,t)=0  \} \subset \{x_n + b(t_0)  \leq C (\lambda |t|+\varepsilon_0)   \} \\ &\subset  \{ x_n \leq C(\eta + \lambda |t| + \varepsilon_0)\}
\end{align*}
hence defining $\bar u(x,t)= u(x+b(t_0)e_n, t+t_0)$ this implies
\[
\p_x \{x: \bar u (x,t)=0\} \subset  \{ x_n \leq C(\eta + \lambda|t+t_0| + \varepsilon_0)\} \subset \{ x_n \leq C(\eta + \lambda \eta^2 + \varepsilon_0) \} \subset \{ x_n \leq C(\eta + \varepsilon_0) \}.
\]
Now, in order to apply the construction of the function $w$ as in the first part of the proof we check that an estimate as \eqref{eq:flat2} hold for $\bar u_\tau$:
\begin{align*}
    \p_x\{x:\bar u _\tau(x,t) = 0\} \subset \{ \tau^{-1}x_n : \bar u (x,\tau^2 t) = 0\} \subset \{ \tau^{-1} x_n \leq C(\eta + \varepsilon_0)\} = \{ x_n \leq C(\eta + \varepsilon_0)\tau\}.
\end{align*}
Thus replicating the argument for $\bar u$ produces
\begin{equation}
    |\bar u (y,s)- \bar a (t_0) y_n| \leq C(1+\lambda \| f^-\|_\infty)\tau^2 \quad \text{in }B_{2\eta} \times [-\eta^2,0]
\end{equation}
and producing the analogous lower bound taking $\bar x_0 = x_0 -b(t_0)e_n$ so that the nondegeneracy condition \eqref{eq:nondegscal} holds for $\bar u$.
Thus similar bounds for $\bar a(t_0)$ are obtained, namely $c(1+\|f^- \|_{\infty}) \leq \bar a(t_0) \leq C$, and for $(x,t) \in B_{2\eta} \times [-\eta^2,0]$ we have
\[
 (\bar a(t_0) x_n - C (1+\lambda\| f^-\|_\infty)\eta^{1+ \gamma})^+ \le \bar u(x,t) \le  (\bar a(t_0) x_n + C (1+\lambda\| f^-\|_\infty)\eta^{1+ \gamma})^+.
\]
Now, from the change of variables $y= x+b(t_0)e_n, \, s=t+t_0$ and the definition of $\bar u$ we get 
\begin{equation} \label{eq:bound1}
(\bar a(t_0) (y_n -b(t_0)) - C (1+\lambda\| f^-\|_\infty)\eta^{1+ \gamma})^+ \le u(y,s) \le  (\bar a(t_0) (y_n -b(t_0)) + C (1+\lambda\| f^-\|_\infty)\eta^{1+ \gamma})^+.  
\end{equation}
in the domain  $B_{2\eta}(b(t_0)e_n) \times [t_0-\eta^2,t_0]$ for any $t_0 \in (-\lambda^{-1}\eta,0)$.
Exploiting the lower bound on $\bar a(t_0)$ and $\lambda \leq 1$ ensures that for some $\tilde C>0$ universal
\[
\frac{C (1+\lambda\| f^-\|_\infty)}{\bar a(t_0)} \leq  \frac{C (1+\lambda\| f^-\|_\infty)}{c (1+\| f^-\|_\infty)} \leq \frac{C}{c}:=\tilde C.
\]
This observation allows to write \eqref{eq:bound1} as
\begin{equation} \label{eq:bound1}
\bar a(t_0) (y_n -b(t_0) - \tilde C \eta^{1+ \gamma})^+ \le u(y,s) \le  \bar a(t_0) (y_n -b(t_0) + \tilde C \eta^{1+ \gamma})^+.  
\end{equation}
The bounds on $u$ above imply that $\bar a(t)$ can vary at most $C \eta^{\gamma}$ in an interval of length $\eta^2$. In fact, for $t \in (t_0 - \eta^2,t_0)$, choosing $y=(b(t_0)+\eta)e_n \in B_{2\eta}(b(t_0))$ and $s \in (t_0-\eta^2,t)$ from \eqref{eq:bound1} we obtain
\begin{align*}
    \bar a(t_0) \eta - C (1+\lambda\| f^-\|_\infty) \eta^{1+\gamma} &\leq u((b(t_0)+\eta)e_n,s) \leq  \bar a(t)(b(t_0)-b(t)+\eta) + C (1+\lambda\| f^-\|_\infty) \eta^{1+\gamma}  \\
     (\bar a(t_0)  - \bar a(t) )\eta &\leq 2C(1+\lambda\| f^-\|_\infty)\eta^{1+\gamma} +a(t)(b(t_0)-b(t)) \leq 2C(1+\lambda\| f^-\|_\infty)\eta^{1+\gamma} \\
      (\bar a(t_0)  - \bar a(t) ) &\leq 2C(1+\lambda\| f^-\|_\infty)\eta^{\gamma}
    \end{align*}
    since $b(t) \geq b(t_0)$ for $t<t_0$. Similarly in the other case, choosing $y=(b(t)+\eta)e_n$ and $s \in (t_0-\eta^2,t)$, we have
    \begin{align*}
        (\bar a(t) - \bar a(t_0))\eta &\leq \bar a(t_0)(b(t)-b(t_0)) + 2C(1+\lambda\| f^-\|_\infty)\eta^{1+\gamma} \\
        \bar a(t) - \bar a(t_0) &\leq C\frac{b(t)-b(t_0)}{\eta} + 2C(1+\lambda\| f^-\|_\infty)\eta^{\gamma} \leq 3C(1+\lambda\| f^-\|_\infty)\eta^{\gamma}
    \end{align*}
where, in the last inequality, we have used \eqref{eq:loc1} for $v(x,t)=u(x-b(t_0)e_n,t-t_0)$ to estimate
\[
b(t)-b(t_0) \leq C|t-t_0| \leq C \tau^2 = C \eta^{1+\gamma}.
\]
We can regularize $\bar a(t)$ by 
averaging over such intervals convolving with a suitable mollifier and the bounds for $u$ still hold after changing the value of the constant $C$. Indeed defining 
\[
\bar a_\eta(t) = \left(\rho\left(\frac{2}{\eta^2}(\cdot - t_0)+1\right) * \bar a\right) (t)
\]
where $\rho$ is the usual mollifier, we have 
\begin{align*}
\bar a_{\eta}'(t) &=  \frac{1}{\eta^2} \int_{t-t_0}^{t-(t_0-\eta^2)} \rho'\left(\frac{2}{\eta^2} (t-t_0-s) +1 \right) \frac{2}{\eta^2} \bar a(s) \, d s 
\intertext{taking $r=\frac{2}{\eta^2} (t-t_0-s) +1$ and $c_1= \int_{-1}^{0} \rho'(r) \, dr = -\int_{0}^{1} \rho'(r) \, dr $}
&=  \frac{1}{\eta^2} \int_{-1}^{0} \rho'(r) \bar a \left( t-t_0-\frac{\eta^2}{2}(r-1)\right)\, dr + \frac{1}{\eta^2} \int_{0}^{1} \rho'(r) \bar a \left( t-t_0-\frac{\eta^2}{2}(r-1)\right)\, dr  \leq \\
& \leq  \frac{1}{\eta^2} \sup_{ (t-t_0,t-t_0+\eta^2)} \bar a \int_{-1}^{0} \rho'(r) \, dr +  \frac{1}{\eta^2} \inf_{ (t-t_0,t-t_0+\eta^2)} \bar a \int_{0}^{1} \rho'(r) \, dr = \frac{c_1}{\eta^2} \osc_{ (t-t_0,t-t_0+\eta^2)} \bar a
 \leq C \eta^{\gamma - 2}
\end{align*}
and from the convergence of $\bar a_{\eta}$ with $\bar a$ we have 
\begin{equation} \label{eq:va2}
\bar a_\eta(t) \left( x_n - b(t) - C\eta^{1+\gamma} \right)^+ \leq u(x,s) \leq \bar a_\eta(t) \left( x_n - b(t) + C\eta^{1+\gamma} \right)^+  \quad \text{in } B_{2C_0\eta} \times [t-\eta^2,t] 
\end{equation}
with the following bounds on the coefficients  
\[
c(1+\|f^- \|_{\infty}) \leq \bar a_\eta(t) \leq C, \quad \nor{\bar a_\eta'(t)} \leq C \eta^{\gamma-2}, \quad \nor{b(t)} \leq C_0 \eta.
\]
In the following $\bar a_\eta$ will be denoted with $\bar a$.

Now, in order to conclude the argument we perturb $b$ by constructing $\Tilde{b}$ which satisfies the ODE $\Tilde{b}'=-\lambda \bar a$ and $\Tilde{b}(0)=0$. In particular we will show that for $t \in [-\lambda^{-1} \eta, 0]$,
\begin{equation} \label{eq:bperturb}
    |b(t) - \Tilde{b}(t) | \leq C \eta^{1+\frac{\gamma}{2}}.
\end{equation}
The method consists in perturbing the family of evolving planes $\bar a(t)(x_n - \Tilde{b}(t))^+$ into a subsolution/supersolution. For this we define
\[
d(t):= \Tilde{b}(t) + C_1 \eta^{\frac{\gamma}{2}} \lambda t,
\]
with $C_1$ large to be fixed. Now, obtaining $\eqref{eq:bperturb}$ is equivalent to show that 
\begin{equation} \label{eq:th}
    b(t) \geq d(t) - 2 \eta^{1+\frac{\gamma}{2}}.
\end{equation}
In this direction we define the function 
\[
v(x,t):= (1-C_2\eta^{\frac{\gamma}{2}})\bar a(t) \left(h(x-d(t)e_n) \right)^+ \quad \text{with}\quad h(x):= x_n - \eta^{\frac{\gamma}{2}-1} \left( \nor{x'}^2 - C_3 x_n^2 \right)
\]
where the constants $C_2$ and $C_3$ are large and will be assigned later in the proof. Choose the constant $C_2$ universal in such a way that 
\begin{equation} \label{eq:va}
    v \leq \bar a(t)(x_n - d(t))^+,
\end{equation}
they also coincide in $d(t)e_n$ by construction and moreover, when $x \in \p B_{2\eta}(d(t) e_n) \cap \{v >0 \}$, the difference is greater than $\eta^{1+\frac{\gamma}{2}}.$ Next, we can check that $v$ is a strict subsolution of \eqref{eq:StefanPara} in the domain $\Omega:= \cup_{t \in [-\lambda^{-1} \eta, 0]} B_{2\eta}(d(t)e_n) \times \{ t\}$.  First,  in $\{v >0 \}$, we have
\begin{align*}
    d'&=\lambda(-\bar a + C_1 \eta^{\frac{\gamma}{2}}) \qquad \nabla h = (-2\eta^{\frac{\gamma}{2}-1}x',1+2C_3\eta^{\frac{\gamma}{2}-1}x_n) \\
    \p_t v&= (1-C_2 \eta^{\frac{\gamma}{2}})\left[\bar a' h +\bar a h_nd' \right]\\ 
    &=(1-C_2 \eta^{\frac{\gamma}{2}}) \left[\bar a' h-\bar a (1+2C_3 \eta^{\frac{\gamma}{2}-1}(x_n - d) \lambda(-\bar a +C_1\eta^\frac{\gamma}{2})\right]  \\ 
    \mathcal{M}_K^{+} (D^2 v) &= (1-C_2 \eta^\frac{\gamma}{2}) \bar a \left( 2KC_3 \eta^{\frac{\gamma}{2}-1} - 2K^{-1}(n-1)\eta^{\frac{\gamma}{2}-1}  \right).
\end{align*}
We obtain the following bound 
$$ |\p_t v| \leq C |\bar a'|\eta +C |d'| \leq C \eta^{\gamma-1}.$$
Hence choosing $C_3$ large enough, depending on $f$, we have that
\begin{align*}
\p_t v-  \mathcal{M}_K^{+} (D^2 v) &\leq C \eta^{\gamma-1} -  \frac{c}{2}\eta^{\frac{\gamma}{2}-1}(2KC_3-2K^{-1}(n-1)) \\
&\leq \eta^{\frac{\gamma}{2}-1} \left( C \eta^{\frac{\gamma}{2}}-\frac{c}{2}KC_3\right) \leq -\frac{c}{4}KC_3 \eta^{\frac{\gamma}{2}-1} < -\lambda f^-. 
\end{align*}

Regarding the condition on the free boundary,
$$v_t = (1-C_2 \eta^{\frac{\gamma}{2}})\bar a (-d ') h_n, \quad \quad |\nabla v|^2 \ge (1-2C_2 \eta^{\frac{\gamma}{2}})\bar a^2 .$$
Since 
$$h_n=1+ O(\eta^{\frac{\gamma}{2}}), \quad \quad \quad (-d')\bar a = \lambda \bar a^2 - C_1 \lambda \bar a \eta^{\frac{\gamma}{2}},$$
 we can choose $C_1> 2C_2C$,  where $C$ is the upper bound of $\bar a$, such that
\begin{align*}
    \p_t v \leq \lambda \bar a^2 \left( 1-\frac{C_1}{\bar a} \eta^{\frac\gamma 2}\right) \leq  \lambda  \bar a^2 \left( 1-\frac{C_1}{C} \eta^{\frac\gamma 2}\right)< \lambda \bar a^2 \left( 1- 2C_2\eta^{\frac\gamma 2}\right) \leq \lambda |\nabla v|^2.
\end{align*}

Arguing by contradiction, if
$$\mbox{$b (t_0)< d(t_0) - 2\eta^{1+\frac{\gamma}{2}}$ for some $t_0 \in [-\lambda^{-1}\eta,0]$, }$$
then by \eqref{eq:va2} and \eqref{eq:va} we find that $v< u$ at time $t=t_0$ in $B_{2\eta}(d(t_0)e_n) \cap \overline {\{v >0\}}$. 
On the other hand $v=u$ at the origin $(0,0)$. 
This means that as we increase $t$ from $t_0$ to $0$, 
the graph of $v (\cdot,t)$ in $B_{2\eta}(d(t)e_n) \cap \overline {\{v >0\}}$ will touch from below the graph of $u$ for a first time $t$, 
and the contact must be an interior point to $B_{2 \eta}(d(t)e_n)$ due to the properties \eqref{eq:va2},\eqref{eq:va} of $u$ and $v$ (in particular the difference between $ \bar a(t) \, (x_n-d(t))^+$ and $v$ is greater than $\eta^{1+\frac{\gamma}{2}}$ on $\p B_{2 \eta}(d(t)e_n)$). 
This contact point is either on the free boundary $\p \{v>0\}$ or on the positivity set $\{v>0\}$ and we reach a contradiction since $v$ is a strict subsolution. The claim \eqref{eq:th} is proved, hence
$$b(t) \ge \tilde b(t) - C \eta^{1+\frac{\gamma}{2}} \quad \mbox{if} \quad  \quad t \in [-\lambda^{-1}\eta,0].$$
The opposite inequality is obtained similarly and the claim \eqref{eq:bperturb} holds. Then from \eqref{eq:va2} and \eqref{eq:bperturb}  we deduce that for all $\eta \le c$ small
$$ \bar a(t) \left(x_n - \tilde b(t) -\eta^{1+\frac \gamma 4}\right) ^+ \, \, \le \, \, u \, \, \le \, \, \bar a(t) \left(x_n -\tilde b(t) + \eta^{1+\frac \gamma 4} \right)^+ $$
in $B_{\eta} \times [-\lambda^{-1}\eta,0]$ with $\frac \gamma 4=\frac{\alpha_0}{4(2+\alpha_0)} \in (0,1)$ universal and
$$c(1+\|f^- \|_{\infty}) \le \bar a(t) \le C, \quad |\bar a'(t)| \le \eta^{\frac \gamma 4-2}, \quad \quad \tilde b'(t)=- \lambda \bar a(t), \quad \tilde b(0)=0.$$ 
\hfill $\square$

\section{The equivalence of the integral and pointwise nondegeneracy condition} \label{Section:Nondeg}
The nondegeneracy condition considered in this paper is equivalent to the usual one that appears in the literature (for example in \cite{ACS2, DFS23})  under the assumption that the Lipschitz norm of the solution is sufficiently small. We prove one of the implications in the following lemma, since the other implication is immediate.

\begin{Lemma}
    If $u$ is a Lipschitz continuous function with respect to the parabolic distance, that is,
    \[
    \sup_{\substack{(x,t),(y,s) \in B_1 \times [-1,0] }}|u(x,t)-u(y,s)|\leq Ld_p((x,t),(y,s))
    \]
     with $L\leq 2K^{-1}(1+\|f^-\|_{\infty} )$ and for every $\lambda>0$ there exists a point $x_0\in\p B_{\frac{3}{4}\lambda}$ such that for every $t_0 \in (-(K\lambda)^{-1},0)$
    \[
    u(x_0,t_0)\geq K^{-1}(1+\|f^-\|_{\infty} ) \lambda,
    \]
    then for every $0<p_0<1$ and $r< \frac{1}{4}\lambda$ we have
    \[
    \left(\fint_{B_r(x_0)\times(t_0,t_0 + r^2)} u^{p_0}\,dxdt\right)^\frac{1}{p_0}\geq \frac{K^{-1}}{2} (1+\|f^-\|_{\infty} )\lambda
    \]
    
\end{Lemma}
\begin{proof}
    First we note that $u^{p_0}$ is $C^{0,p_0}$ in the parabolic sense. This is a consequence of the fact that the function
    \[
    g(y)=1-y^{p_0}-(1-y)^{p_0}
    \]
    is convex. Thus we compute 
    \begin{align*}
        u^{p_0}(x_0,t_0)\leq&\, \left|u^{p_0}(x_0,t_0)-\fint_{B_r(x_0)\times(t_0,t_0 +r^2)} u^{p_0}\right|+\left(\fint_{B_r(x_0)\times(t_0, t_0+r^2)} u^{p_0}\right)\\
        \leq&\,\fint_{B_r(x_0)\times(t_0,t_0+ r^2)} |u^{p_0}(x_0,t_0)-u^{p_0}(x,t)|+\left(\fint_{B_r(x_0)\times(t_0,t_0+ r^2)} u^{p_0}\right).
    \end{align*}
    Therefore since $L\leq 2K^{-1}(1+\|f^-\|_{\infty} )$
    \begin{align*}
        \left(\fint_{B_r(x_0)\times(r^2, 2r^2)} u^{p_0}\right)^\frac{1}{p_0}\geq \left[K^{-p_0}(1+\|f^-\|_{\infty})^{p_0}-\left(\frac{L}{4}\right)^{p_0}\right]^\frac{1}{p_0}\lambda \geq\frac{K^{-1}}{2} (1+\|f^-\|_{\infty} )\lambda .
    \end{align*}

\end{proof}
\begin{Remark} Without requiring any further regularity on $u$ it is possible to consider, instead of \eqref{Nondeg_int}, the pointwise nondegeneracy condition with the full $L^{\infty}$-norm of $f$, namely for  every $t_0 \in (-(K\lambda)^{-1},0)$
\begin{equation} \label{Nondeg_point2}
    u(x_0,t_0) \geq K^{-1}(1+\norm{f}_{\infty}) \lambda.
\end{equation}
This can be done replacing the Weak Harnack inequality in \eqref{eq:weak11} with the Harnack inequality in the proof of Theorem \ref{Theorem:Flatsoltoflatfree}.
As discussed in the introduction, the nondegeneracy condition \eqref{Nondeg_int} more accurately captures the main obstacle of the inhomogeneous Stefan problem, since it depends only on the negative part of $f$.
\end{Remark}
\subsubsection*{Acknowledgements}
This research is partially supported by PRIN 2022 7HX33Z - CUP J53D23003610006 and by University of Bologna funds
\say{Attività di collaborazione con università del Nord America} in the framework of the project:
 \say{Interplaying problems in analysis and geometry}. F.F. and D.G. are also partially supported by  INDAM-GNAMPA project 2025: \say{At The Edge Of Ellipticity} - CUP E5324001950001. 
The authors would like to thank Ovidiu Savin and Daniela De Silva for fruitful
conversations on the topic of this paper. 
D.G. and D.J. wish to thank the Department of Mathematics of Columbia University for the warm hospitality.
\bibliographystyle{abbrv}
\bibliography{BibStef.bib}
\Addresses
\end{document}